\documentclass{amsart}

\usepackage{geometry}
\usepackage{amsmath}
\usepackage{amssymb}
\usepackage{amsthm}
\usepackage{mathtools}
\usepackage{tikz}
\usepackage{graphicx}

\usepackage{float}

\theoremstyle{definition}
\newtheorem{thm}{Theorem}[section]
\newtheorem{cor}[thm]{Corollary}

\newtheorem{lem}[thm]{Lemma}

\newtheorem{defn}[thm]{Definition}

\newtheorem{rmk}[thm]{Remark}

\newcommand{\rk}{\text{rk}}


\begin{document}

\title{Symmetric Chain Decompositions of Products of Posets with Long Chains}
\author{Stefan David}
\address{University of Cambridge, Trinity College, Trinity Street, CB21TQ}
\email{sd637@cam.ac.uk}
\thanks{The first author would like to thank Trinity College, Cambridge, for providing travel funding.}

\author{Hunter Spink}
\address{Harvard University, Cambridge, 1 Oxford Street, 02138}
\email{hspink@math.harvard.edu}
\thanks{The second author would like to thank Harvard University for providing travel funding.}

\author{Marius Tiba}
\address{University of Cambridge, Trinity Hall, Trinity Lane, CB21TJ}
\email{mt576@cam.ac.uk}
\thanks{The third author would like to thank Trinity Hall, Cambridge, for support through the Trinity Hall Research Studentship.}

\begin{abstract}
We ask if there exists a symmetric chain decomposition of the cuboid $Q_k \times n$ such that no chain is \textit{taut}, i.e. no chain has a subchain of the form $(a_1,\ldots, a_k,0)\prec \ldots\prec (a_1,\ldots,a_k,n-1)$. In this paper, we show this is true precisely when $k \ge 5$ and $n\ge 3$. This question arises naturally when considering products of symmetric chain decompositions which induce orthogonal chain decompositions --- the existence of the decompositions provided in this paper unexpectedly resolves the most difficult case of previous work by the second author on almost orthogonal symmetric chain decompositions \cite{orth}, making progress on a conjecture of Shearer and Kleitman \cite{SK}. In general, we show that for a finite graded poset $P$, there exists a canonical bijection between symmetric chain decompositions of $P \times m$ and $P \times n$ for $m, n\ge \rk(P) + 1$, that preserves the existence of taut chains. If $P$ has a unique maximal and minimal element, then we also produce a canonical $(\rk(P) +1)$ to $1$ surjection from symmetric chain decompositions of $P \times (\rk(P) + 1)$ to symmetric chain decompositions of $P \times \rk(P)$ which sends decompositions with taut chains to decompositions with taut chains.
\end{abstract}

\maketitle

\section{Introduction}
The consideration of symmetric chain decompositions of posets first started with Kleitman's proof of the Littlewood-Offord theorem on concentration of sums of Bernoulli random variables \cite{Kleitman}. One of the key observations in that paper is that we can inductively create symmetric chain decompositions of the hypercube $Q_n=\{0,1\}^n$ (which can alternatively be viewed as the power set of $\{1,2,\ldots, n\}$), through a certain ``duplication method''. This observation is the special case with $Q$ a two-element chain poset of a more general claim that given two posets $P,Q$ with symmetric chain decompositions, we can decompose the product $P\times Q$ into symmetric chains by decomposing the rectangle posets formed by the product of a chain in $P$ with a chain in $Q$.  The literature is abundant with both necessary and sufficient conditions for the existence of symmetric chain decompositions on finite graded posets (see e.g. \cite{Griggs} and \cite{Stanley2}). However, the further study of commonalities between all symmetric chain decompositions is somewhat lacking, mostly due to the largely unstructured nature of a generic such decomposition of a typical poset.

The first attempt to study multiple symmetric chain decompositions simultaneously on a given poset occurred perhaps in 1979 when Shearer and Kleitman \cite{SK} found the minimum probability that two randomly chosen elements of $Q_n$ are comparable for an arbitrary probability distribution. In their proof, they introduced the notion of ``orthogonal chain decompositions'' of $Q_n$, which are decompositions of $Q_n$ into $n \choose {\lfloor n/2\rfloor}$ chains for which any chain in one decomposition intersects any other chain in the other decomposition in at most one element. Their construction proceeds by slightly modifying two ``almost orthogonal symmetric chain decompositions'' of $Q_n$, isolated in \cite{orth}, which are symmetric chain decompositions that satisfy the orthogonal intersection condition except for the maximal chain in the decompositions, which must intersect in precisely their top and bottom elements. In \cite{SK}, Shearer and Kleitman further conjectured that there exist $\lfloor n/2 \rfloor +1$ orthogonal decompositions of $Q_n$; no progress has been made on the conjecture until \cite{orth} and the present paper. In \cite{orth}, the second author has shown that three orthogonal decompositions can be constructed for all sufficiently high dimensional hypercubes, and they additionally arise from three almost orthogonal symmetric chain decompositions.

The strategy pursued in \cite{orth} is as follows. Suppose that for $1 \le j \le l$, we have almost orthogonal symmetric chain decompositions $\mathcal{F}^j_i$ of $Q_{n_i}$ for $i=1,2,\ldots, r$. Then to create $l$ almost orthogonal symmetric chain decompositions of $Q_{n_1+\ldots + n_r}$, we aim to give symmetric chain decompositions of the cuboids in $\prod_i \mathcal{F}^j_i$ in such a way that the chains inside cuboids in $\prod_i \mathcal{F}^j_i$ and chains inside cuboids in $\prod_i \mathcal{F}^{j'}_i$ intersect in at most one element when $j \ne j'$ (except of course for the two maximal chains, which we require to intersect in just their top and bottom elements).

To put the questions addressed in the present paper in the proper context, we consider the most difficult case from \cite{orth}. For this, it suffices to take $l=2$. Suppose $r=k+1$ and take the product of a $2$-element chain from each $\mathcal{F}^1_i$ for $1 \le i \le k$ with the maximal chain in $\mathcal{F}^1_{k+1}$. Similarly, take the analogous product with the $\mathcal{F}^2_i$. Let $n=n_{k+1}+1$ be the size of a maximal chain in the last hypercube. Using $\textbf{s}$ to refer to the $s$-element chain poset $0 \prec 1 \prec \ldots \prec s-1$, we have two cuboids of the form $P(k,n) = \underbrace{\textbf{2} \times \textbf{2} \times \ldots \times \textbf{2}}_k \times \textbf{n}= Q_k \times \mathbf{n}$, with the property that their intersection is either empty, or is $\{x\} \times \{\min_{Q_{n_{k+1}}},\max_{Q_{n_{k+1}}}\}$, where $x$ is some element of $Q_{n_1}\times \ldots \times Q_{n_k}$. To avoid the situation of having two chains intersect in at least two elements, it suffices to decompose $P(k,n)$ such that no subchain of a chain has the form $(a_1,\ldots, a_k,0)\prec \ldots\prec (a_1,\ldots,a_k,n-1)$. In $P(k,n)$, we call a symmetric chain containing such a subchain \textit{taut}. More generally, given a finite graded poset $P$, we say a symmetric chain in $P \times \textbf{n}$ is $\textit{taut}$ if it contains for some $p \in P$ a subchain of the form $(p, 0) \prec (p, 1) \prec \ldots \prec (p, n-1)$.

From this, the most natural question that arises is whether there is a symmetric chain decomposition of $P(k,n)$ without a taut chain. One of the main results of this paper,  Theorem~3.1, completely answers this question.

Surprisingly, the answer is more interesting than one might initially suspect. Indeed, for the family of posets $P(k,n)$ with $k \le 4$, i.e. for the posets $\textbf{2} \times \textbf{n}$, $\textbf{2} \times \textbf{2} \times \textbf{n}$, $\textbf{2} \times\textbf{2} \times \textbf{2} \times \textbf{n}$, and $\textbf{2} \times\textbf{2} \times\textbf{2} \times \textbf{2} \times \textbf{n}$, every symmetric chain decomposition has a taut chain. For $k\ge 5$ and $n \ge 3$ however, we will explicitly construct in Section 5 decompositions with no taut chains by boot-strapping decompositions of $P(5,3)$, $P(5,4)$ and $P(5,5)$ using more general results about finite graded posets we prove in the remaining sections. These decompositions turn out to be very hard to find, as they are completely ad hoc.

One of the general bootstrapping results we prove is Theorem~3.2, which says that if $P$ is a finite graded poset, there exists a canonical bijection between symmetric chain decompositions of $P \times \mathbf{m}$ and $P \times \mathbf{n}$ for $m, n\ge \rk(P) + 1$, that preserves the existence of taut chains.

Also, we show in Theorem 3.3 that if $P$ has a unique maximal/minimal element, then there is a canonical $(\rk(P)+1)$ to $1$ surjection from symmetric chain decompositions of $P \times \textbf{\rk(P)+1}$ to symmetric chain decompositions of $P \times \textbf{\rk(P)}$ which sends decompositions with taut chains to decompositions with taut chains. Under a mild additional hypothesis, we also show that $P \times \textbf{\rk(P)}$ has a decomposition without taut chains if and only if this is true for $P \times \textbf{(\rk(P)+1)}$. 

The structure of this paper is as follows.
\begin{itemize}
\item In Section 2, we recall some basic definitions about finite graded posets, and some definitions briefly mentioned in the introduction.
\item In Section 3, we state our main results.
\item In Section 4, we prove the main results pertaining to general finite graded posets.
\item In Section 5, we explicitly construct symmetric chain decompositions with no taut chains for $P(3,5)$, $P(4,5)$, and $P(5,5)$. By the results in Section 4, we complete the proof of Theorem 3.1, classifying which $P(k,n)$ have symmetric chain decompositions without taut chains.
\end{itemize}

\section{Definitions}

We first recall some basic definitions about finite graded posets.

\begin{defn}
In a poset $P$, we say that $y$ \textit{covers} $x$, denoted $x \prec y$, if $x < y$ and there is no element $z$ such that $x<z<y$. A finite graded poset $(P, \le)$ is a finite poset equipped with a \textit{rank function} $\rk:P \to \mathbb{N}$ such that the rank of every minimal element is $0$, and if $x \prec y$ then $\rk(y)=\rk(x)+1$. The rank of $P$, denoted by $\rk(P)$, is the maximal value of $\rk$ on $P$. Say $P$ is \textit{rank-symmetric} if the number of elements of rank $r$ is the same as the number of elements of rank $\rk(P)-r$. If $P$ has a unique maximal/minimal element, then we will denote them by $\min_P$ and $\max_P$.
\end{defn}

All posets in this paper are finite graded posets; $P$ will always refer to a finite graded poset.

\begin{defn}
A \textit{symmetric chain} in $P$ is a chain which for some $r$ consists of exactly one element of ranks $r,r+1,\ldots, \rk(P)-r$. A \textit{symmetric chain decomposition} of $P$ is a partition of $P$ into symmetric chains.
\end{defn}

The following notion of ``tautness'' (discussed in the previous section) is the key notion studied in this paper.

\begin{defn}
We define a symmetric chain in $P \times \mathbf{n}$ to be \textit{taut} if it contains a subchain of the form $$p \times \textbf{n}=(p,0)\prec \ldots \prec (p,n-1)$$ for some $p \in P$.
\end{defn}

Recall the poset defined in the introduction $$P(k,n)=Q_k \times \textbf{n},$$ where $Q_k$ is the $k$-dimensional hypercube. 

For $P(k,n)$, the rank function $\rk$ is simply the sum of the coordinates. In particular, $$\rk(P(k,n))=k+n-1.$$

\section{Main Results}

The central result of this paper is Theorem 3.1, proved in Section~5.

\begin{thm}
There exists a symmetric chain decomposition of $P(k,n)$ with no taut chain if and only if $k \ge 5$ and $n \ge 3$.
\end{thm}

Importantly, for fixed $Q_k$, making $n$ very large does not aid us in constructing decompositions with no taut chains. If $n$ is very large, one might hope that for each chain, the poset provides enough mobility such that we could make the chain turn in one of the many $Q_k$'s down the line before continuing on, thus breaking tautness. The phenomenon preventing us from doing that however is that once $n$ is large enough, the number of chains is equal to the number of elements in $Q_k$, so we ``clog up'' most of the $Q_k \times \{r\}$'s with chains, preventing us from turning.

Most of our considerations generalize under mild conditions to arbitrary posets $P$ in place of $Q_k$. In particular, we prove the following two theorems which we later apply to $Q_k$ in the proof of Theorem~3.1. These theorems would allow one to answer the analogous question for $P\times \mathbf{n}$ in a similar way, reducing the problem to a finite computation.

\begin{thm}
In a rank-symmetric poset $P$, if $m,n \ge \rk(P)+1$, then there is a canonical bijection between the set of symmetric chain decompositions of $P\times \mathbf{m}$ and of $P \times \mathbf{n}$ which bijects decompositions with taut chains.
\end{thm}

\begin{thm} Let $P$ be a rank-symmetric poset with a unique maximum and minimum element. Then there is a $(\rk(P) + 1)$ to $1$ surjection from the set of symmetric chain decompositions of $P\times \textbf{(\rk(P)+1)}$ to the set of symmetric chain decompositions of $P \times \textbf{rk(P)}$ which sends decompositions with taut chains to decompositions with taut chains. Furthermore, if $\max_P$ covers at least $2$ elements, then $P \times \textbf{\rk(P)}$ has a decomposition without taut chains if and only if this is true for $P \times \textbf{(\rk(P)+1)}$.
\end{thm}
\begin{rmk}
The hypothesis of $\max_P$ covering two elements in Theorem 3.3 is needed for example when $P=\mathbf{3}$.
\end{rmk}

\section{Proofs of general results}
In this section we prove Theorems 3.2 and 3.3, as well as some auxiliary results used in proving Theorem 3.1.
\begin{defn}
In the poset $P \times \mathbf{n}$, we define a \textit{packet} to be the collection of elements of a given rank and $\mathbf{n}$-coordinate. The \textit{rank} of a packet $\Lambda$, denoted $\rk(\Lambda)$, is the common rank of elements in $\Lambda$. It will be very convenient for us to reparametrize $P \times \mathbf{n}$, and define $[p,r]:=(p,r-\rk(p)) \in P \times \mathbf{n}$ whenever $\rk(p) \le r < \rk(p)+n$. This makes $[p,r]$ is the unique element of $P \times \mathbf{n}$ with $P$-coordinate $p$ and rank $r$.
\end{defn}
Consider the map $P \times \textbf{n} \to \mathbb{Z}^2$ given by $(p,r) \mapsto (\rk(p),\rk(p)+r)$. Note that $\rk((p,r))=\rk(p)+r$ in $P \times \textbf{n}$, so $[p,r]$ gets sent to $(\rk(p),r)$. Under this map, the elements of $P \times \textbf{n}$ are identified into their packets. If we label each point in the image of this map with the number of points in the corresponding packet, we call this the \textit{pictorial representation} of $P \times \textbf{n}$. Figure 1 depicts the pictorial representation of $Q_4 \times \textbf{6}$ --- we use the pictorial representation for many of our proofs.

In the pictorial representation, the row $y=k$ contains all elements of rank $k$ in $P \times \textbf{n}$, the column $x=l$ contains all elements with $P$-coordinate of rank $l$, and the diagonal $y=x+r$ contains all elements of $\textbf{n}$-coordinate $r$. A chain which skips no ranks connects a sequence of packets, with each packet following the previous one either vertically up one packet (which is uniquely determined in $P \times \textbf{n}$ as increasing the $\mathbf{n}$-coordinate by 1), or diagonally up-right one packet (which corresponds to moving up in the $P$-coordinate).

\begin{figure}[H]
\centering
\begin{tabular}{c c c c c}
&&&&1\\
&&&4&1\\
&&6&4&1\\
&4&6&4&1\\
1&4&6&4&1\\
1&4&6&4&1\\
1&4&6&4&\\
1&4&6&&\\
1&4&&&\\
1&&&&
\end{tabular}
\caption{Pictorial Representation of $P \times \textbf{n}$ in the case of $Q_4 \times \mathbf{6}$}
\end{figure}

We observe that for $n \ge \rk(P)$, the pictorial representation of $P \times \mathbf{n}$ has rows with packets at coordinates $(x,y)$ as follows.
\begin{itemize}
\item For $0 \le y < \rk(P)$, there are $y+1$ packets with $0 \le x \le y$.
\item For $\rk(P) \le y \le n-1$, there are $\rk(P)+1$ packets with $0 \le x \le \rk(P)$.
\item For $n-1 < y \le \rk(P)+n-1$, there are $\rk(P)+n-y$ packets with $y-(n-1) \le x \le \rk(P)$.
\end{itemize}
We refer to these collections of rows, or equivalently these collections of ranks in the poset $P \times \mathbf{n}$, as the \textit{first block}, \textit{middle block}, and \textit{last block}, respectively. We make the important observation that for $m,n \ge \rk(P)$, the subposets corresponding to the first blocks of $P \times \mathbf{m}$ and $P \times \mathbf{n}$ are canonically isomorphic, as are the subposets corresponding to the last blocks.
\begin{lem}
If $P\times \mathbf{n}$ has a decomposition with no taut chain, and $Q$ is a poset with a symmetric chain decomposition, then $(P\times Q) \times \mathbf{n}$ has a decomposition with no taut chains.
\end{lem}

\begin{proof}
We take the product of each non-taut chain in $P \times \mathbf{n}$ with each chain in the symmetric chain decomposition of $Q$ and decompose each resulting rectangle into symmetric chains arbitrarily. Then the resulting chains are symmetric, non-taut, and give a symmetric chain decomposition of $(P \times Q) \times \mathbf{n}$ as desired.
\end{proof}
\begin{cor} If $P(k,n)$ has a symmetric chain decomposition with no taut chain, then so does $P(k',n)$ for any $k' \ge k$.
\end{cor}
\begin{lem}
If $P \times \mathbf{n}$ has a symmetric chain decomposition, then $P$ must be rank-symmetric.

If furthermore $P \times \mathbf{n}$ has a symmetric chain decomposition into non-taut chains, then \begin{itemize}
\item if $\rk(P)$ is even, the size of the middle rank of $P$ does not exceed the sum of all the sizes of lower ranks, and
\item if $\rk(P)$ is odd, the common size of the middle ranks of $P$ does not exceed twice the sum of the sizes of all ranks strictly before the middle ranks.
\end{itemize}
\end{lem}
\begin{proof}
If $P \times \mathbf{n}$ has a symmetric chain decomposition, then by the rank-symmetry of $P \times \textbf{n}$, we find that $P$ is rank-symmetric (by arguing inductively from the smallest rank up).

Suppose $\rk(P)$ is even and we have a decomposition of $P \times \mathbf{n}$ into non-taut chains. Note that $\rk(P \times \mathbf{n})=\rk(P)+\rk(\mathbf{n})=\rk(P)+n-1$. Let $\Lambda$ be the packet of elements $(p,n-1)\in P \times \textbf{n}$ with $\rk(p)=\rk(P)/2$. As $\rk(\Lambda)+\rk(P)/2=\rk(P\times \mathbf{n})$, a symmetric chain which contains an element $(p, n-1)\in\Lambda$ must also contain an element of the form $[q_p,\rk(P)/2]$. As the chains are disjoint, the elements $[q_p,\rk(P)/2]$ are distinct, so the $q_p$ are distinct.

We have $\rk(q_p)< \rk(P)/2$, as if $\rk(q_p)=\rk(P)/2$, then the $\textbf{n}$-coordinate of this element is $0$ and $p=q_p$, so the symmetric chain is taut. Hence, as the $q_p$'s are distinct, if the number of elements $p$ of middle rank in $P$ exceeds the number of elements of lower rank, then there are not enough elements $q_p$ to accommodate the chains passing through elements of $\Lambda$.

Finally, if $\rk(P)$ is odd, we apply the above argument to the even-ranked poset $P \times \mathbf{2}$ (using Theorem 2.2 with $Q=\mathbf{2}$).
\end{proof}

\begin{cor}
If $k \le 4$ or $n \le 2$, then every symmetric chain decomposition of $P(k,n)$ contains a taut chain.
\end{cor}
\begin{proof}
For $n=1$ the result is trivial, and for $n=2$ the maximal chain is always taut. For $k=1,2$, the result is trivial by inspection. For $k=3,4$, Lemma~4.4 applies.
\end{proof}
Note that for $P=Q_k$ with $k\ge 5$, Lemma~4.4 does not apply. Now we are ready to prove Theorems 3.2 and 3.3.

\begin{proof}[Proof of Theorem 3.2]
The key observation is that when $m,n \ge \rk(P)+1$, the ``clogging up'' behavior discussed after Theorem 3.1 forces the chains in the middle block to move in a certain way that allows us to add or subtract rows in the middle block without changing the nature of the symmetric chain decomposition.

To be more precise, as we mentioned when we introduced the pictorial representation, if $n \ge \rk(P)+1$, we have $n-\rk(P)$ consecutive rows in the middle block at $y=\rk(P),\rk(P)+1,\ldots, n-1$, each consisting of $\rk(P)+1$ packets with coordinates $(x,y)$ with $x=0,1,\ldots, \rk(P)$. Furthermore, in these consecutive rows, for a fixed $x$, the number of elements in the packets with coordinates $(x, \rk(P)),\ldots, (x, n-1)$ are the same. Each of these rows corresponds to a rank in $P \times \textbf{n}$, and the above discussion shows that these ranks have the same number of elements, so a symmetric chain decomposition when restricted to any pair of adjacent rows must biject the elements between them.

As the chains can only move vertically up and diagonally up-right, and any two of these rows have identical packet sizes, this bijection is clearly only possible by having all of the chains move vertically up across this block of rows.

Now if $m \ge \rk(P)+1$ as well, we can modify a symmetric chain decomposition for $P \times \textbf{n}$ to create one for $P \times \textbf{m}$ as follows. Write each chain $C$ in the decomposition of $P \times \textbf{n}$ as the disjoint union of chains $C_1 \cup C_2 \cup C_3$, where $C_2$ is the subchain of elements in the middle block, of the form $[p,\rk(P)] \prec [p,\rk(P)+1] \prec \ldots \prec [p,n-1]$, $C_1$ is the subchain of elements in the first block, and $C_3$ is the subchain of elements in the last block. We modify $C$ to become a chain in $P\times \textbf{m}$ by replacing $C_2$ with $[p,\rk(P)] \prec [p,\rk(P)+1] \prec \ldots \prec [p,m-1]$, and shifting $C_3$ by adding $m-n$ to the last coordinate of each element in $C_3$.

Finally, it is easy to see that this process preserves tautness of chains between $P \times \textbf{n}$ and $P \times \textbf{m}$.

\end{proof}

\begin{proof}[Proof of Theorem 3.3]
We first note that $P \times \textbf{(rk(P)+1)}$ has exactly one row with $\rk(P)+1$ packets in the middle block, and $P \times \textbf{\rk(P)}$ has an empty middle block. These posets have identical first blocks and identical last blocks, so the only difference between them is that the first and last blocks join up directly in the case of $P \times \textbf{\rk(P)}$, whereas in $P \times \textbf{(\rk(P)+1)}$, there's a single row with $\rk(P)+1$ packets separating them. In this proof, we carry out a very similar procedure to the proof of Theorem 3.2, and carefully analyze what happens around the middle block.

In the pictorial representation of $P \times \textbf{(rk(P)+1)}$, call $M$ the row in the middle block with packets at $(0, \rk(P)),\ldots (\rk(P),\rk(P))$, $M^{-}$ the row right below $M$ with packets at $(0, \rk(P)-1), \ldots (\rk(P)-1,\rk(P)-1)$, and $M^{+}$ the row right above $M$ with packets at $(1, \rk(P)+1), \ldots, (\rk(P),$\\$ \rk(P)+1)$. From the locations of the packets, the number of elements in the packets in $M$ is $1$ more than that in $M^{-}$ and that in $M^{+}$, as $P$ has a unique maximum and minimum element. Hence, there is a unique chain of length $1$ in $M$ in some packet $\Lambda$, and the symmetric chain decomposition bijects the remaining elements in $M$ with those in $M^{-}$ and those in $M^{+}$. By working from left to right in $M$, we get that the numbers of chains connecting pairs of packets from $M^{-}$ to $M$ are all completely determined by $\Lambda$: all packets in $M$ which are strictly to the right of $\Lambda$ receive precisely one chain diagonally from $M^{-}$, and all other chains between $M^{-}$ and $M$ are vertical. Similarly, we get all packets in $M$ which are strictly to the left of $\Lambda$ send one chain diagonally to $M^{+}$, and all other chains between $M$ and $M^{+}$ are vertical.

Hence, every element in $M^{-}$ is connected to an element in $M^{+}$ whose $P$-coordinate has rank at most $1$ higher. We can thus modify a symmetric chain decomposition of $P \times \textbf{(\rk(P)+1)}$ to one of $P \times \textbf{\rk(P)}$ as follows. Ignore the chain of length $1$, and for every other chain, decompose it as $D_1 \cup D_2 \cup D_3$, where $D_2$ is the element in the middle block, $D_1$ is the subchain of elements in the first block, and $D_3$ is the subchain of elements in the last block. To construct the chain in $P \times \textbf{\rk(P)}$, we remove $D_2$, and decrease the second coordinate of all elements of $D_3$ by $1$.

It is easy to check if a chain was taut, then it remains taut, and all newly constructed chains are still symmetric chains in $P \times \textbf{\rk(P)}$.

We now verify that the map above from the set of symmetric chain decompositions of $P \times (\textbf{\rk(P)+1})$ to the set of symmetric chain decompositions of $P \times (\textbf{\rk(P)})$ is a $(\rk(P)+1)$ to $1$ surjection. Suppose we have a symmetric chain decomposition $\mathcal{S}$ of $P \times \textbf{\rk(P)}$, viewed as a directed graph via $\prec$. Denote by $N^{-}$ and $N^{+}$ the two middle rows in the pictorial representation of $P \times \textbf{\rk(P)}$ (the last row of the first block, and the first row of the last block). Consider the directed graph $G$ (with loops) on $P$ defined by taking the restriction of $\mathcal{S}$ to $N^{-} \cup N^{+}$, and projecting this induced directed subgraph onto the $P$-coordinate. As $\mathcal{S}$ induces a bijection between $N^{-}$ and $N^{+}$, all vertices in $G$ except $\min_P$ and $\max_P$ have in-degree and out-degree $1$. Also, $\min_P$ has out-degree $1$ and in-degree $0$, while $\max_P$ has in-degree $1$ and out-degree $0$. Every directed edge in $G$ is either a loop, or increases rank by $1$ in $P$. From this, we deduce that $G$ consists of one directed maximal chain (from $\min_P$ to $\max_P$) and loops on the remaining vertices.

We show now that there exists a canonical equivalence between symmetric chain decompositions $\mathcal{S}'$ of $P \times \textbf{(\rk(P)+1)}$ that are mapped to $\mathcal{S}$, and matchings $f$ between the edges of $G$ and their endpoints.

Set $m^{-},m,m^{+},n^{-},$ and $n^{+}$ to be the ranks of $M^{-},M,M^{+},N^{-},$ and $N^{+}$ respectively ($n^{+}-1=n^{-}=m^{-}=m-1=m^{+}-2$). 

First, supposing that we have such a matching $f$, we construct $\mathcal{S}'$ as follows. Identify the restriction of $\mathcal{S}'$ in the first block of $P \times \textbf{\rk(P)+1}$ with the restriction of $\mathcal{S}$ in the first block of $P \times \textbf{\rk(P)}$, and similarly for the last blocks. All that remains now is to identify the $1$-element chain in $P \times \textbf{\rk(P)+1}$, and to correctly join up the ends of the chains in $M^{-}$ with the starts of the chains in $M^{+}$. Consider a directed edge $e$ from $p$ to $q$ in $G$, corresponding to $[p,n^{-}] \prec [q,n^{+}]$ from $N^{-}$ to $N^{+}$ in $P\times \textbf{\rk(P)}$. Then we create the chain  $[p,m^{-}] \prec [f(e),m] \prec [q,m^{+}]$ from $M^{-}$ to $M^{+}$ in $P \times \textbf{(\rk(P)+1)}$. Finally, there is a unique vertex $v$ in $G$ which no edge matches to, from which we create the $1$-element chain $[v,m]$ in $\mathcal{S}'$. All of these chains do not intersect, as $f$ is an injection, and $f$ misses $v$. Also, all chains in $\mathcal{S}'$ are symmetric.

Conversely, supposing we have a symmetric chain decomposition $\mathcal{S}'$ of $P \times \textbf{(\rk(P)+1)}$ which maps to $\mathcal{S}$, we construct a matching between the edges of $G$ and their endpoints as follows. Given a directed edge $e$ from $p$ to $q$ in $G$ corresponding to $[p,n^{-}]\prec [q,n^{+}]$ in $P \times \textbf{\rk(P)}$, consider the chain in $\mathcal{S}'$ which connects $[p,m^{-}]$ to $[q,m^{+}]$ in $P \times (\textbf{\rk(P)+1})$. We define $f(e)$ so that $[f(e),m]$ is the intermediate point on this chain. Clearly this is a matching, as the chains between $M^{-}$ and $M^{+}$ are disjoint, and $\rk(q)$ is at most $1$ higher than $\rk(p)$ so $f(e)=p$ or $q$.

These two maps are inverses of each other, proving the equivalence. As there are $\rk(P)+1$ matchings on $G$ (coming from the $\rk(P)+1$ possible matchings on the edges of the long chain in $G$), the conclusion follows.

Finally, suppose we have a symmetric chain decomposition of $P \times \textbf{(\rk(P)+1)}$ with no taut chain, and $\max_P$ covers at least $2$ elements. A taut chain in $P \times \textbf{rk(P)}$ is created in exactly the following cases. Either the maximal chain in $P \times \textbf{(\rk(P)+1)}$ has a subchain of the form $(\min_P,0) \prec (\min_P,1) \prec \ldots \prec (\min_P, \rk(P)-1)$, or a subchain of the form $(\max_P,1) \prec (\max_P,2) \prec \ldots \prec (\max_P,\rk(P))$. Disconnect $(\min_P,0)$ and $(\max_P,\rk(P))$ from the maximal chain. Connect $(\max_P,\rk(P))$ to an adjacent element with second coordinate also $\rk(P)$ which does not belong to the chain containing $(\min_P,\rk(P)-1)$ (this is possible as there are at least $2$ choices by the hypothesis on $P$), and add a connection from $(\min_P,0)$ to the chain which  $(\max_P,\rk(P))$ now belongs to. This new configuration of symmetric chains now avoids the two cases which would cause taut chains to appear in $P \times \textbf{\rk(P)}$, without creating any taut chains in $P \times \textbf{(\rk(P)+1)}$. This finishes the proof.

\end{proof}

\section{Proof of Theorem 3.1}
By Corollary 4.5, we only have left to construct symmetric chain decompositions of $P(k,n)$ for $k \ge 5$, and $n \ge 3$ with no taut chains. In the tables at the end of the paper, we give decompositions with no taut chains for $k=5$, $n=3,4,5$. Theorem 3.3 then yields such a decomposition for $k=5$, $n=6$, and Theorem 3.2 then yields such a decomposition for $k=5$ and all $n \ge 3$. Finally from this, Corollary 4.3 can then be used to get such decompositions for all $k \ge 5$ and $n \ge 3$.

In tables 1,2,3, the rows give the symmetric chains in $Q_5 \times \textbf{n}$, written in coordinates. Aiding in the finding of the decompositions below were the packet descriptions, and the natural $\mathbb{Z}/5\mathbb{Z}$ action on the points of $Q_5 \times \textbf{n}$.

\begin{table}
\centering
\begin{tabular}{c c c c c c c c c}
1& & &110000&111000&111100&111110& &\\
2& & &011000&011100&011110&011111& &\\
3& & &001100&001110&101110&101111& &\\
4& & &000110&100110&110110&110111& &\\
5& & &100010&110010&111010&111011& &\\
6&000000&100000&101000&101001&111001&111101&111111&111112\\
7& &010000&010100&010101&011101&011102&011112&\\
8& &001000&001010&001011&001111&001112&101112&\\
9& &000100&100100&100101&100111&100112&110112&\\
10& &000010&010010&010011&110011&110012&111012&\\
11& &000001&100001&110001&110002&110102&111102&\\
12& & &010001&011001&011002&011012& &\\
13& & &001001&001101&001102&101102& &\\
14& & &000101&000111&000112&010112& &\\
15& & &000011&100011&100012&101012& &\\
16& & &000002&100002&101002&111002& &\\
17& & & &010002&010102& & &\\
18& & & &001002&001012& & &\\
19& & & &000102&100102& & &\\
20& & & &000012&010012& & &\\
21& & & &110100&110101& & &\\
22& & & &011010&011011& & &\\
23& & & &101100&101101& & &\\
24& & & &010110&010111& & &\\
25& & & &101010&101011& & &\\
\end{tabular}
\caption{Symmetric chain decomposition of $P(5,3)$ with no taut chains}
\end{table}

\begin{table}
\centering
\begin{tabular}{c c c c c c c c c c}
1&000000&100000&101000&101100&101101&101102&111102&111103&111113\\
2& &010000&010100&010110&010111&010112&011112&011113&\\
3& &001000&001010&101010&101011&101012&101112&101113&\\
4& &000100&100100&110100&110101&110102&110112&110113&\\
5& &000010&010010&011010&011011&011012&111012&111013&\\
6& & &110000&110001&110002&111002&111003& &\\
7& & &011000&011001&011002&011102&011103& &\\
8& & &001100&001101&001102&001112&001113& &\\
9& & &000110&000111&000112&100112&100113& &\\
10& & &100010&100011&100012&110012&110013&\\
11& & & &111000&111100&111110& & &\\
12& & & &011100&011110&011111& & &\\
13& & & &001110&101110&101111& & &\\
14& & & &100110&110110&110111& & &\\
15& & & &110010&111010&111011& & &\\
16& &000001&100001&101001&111001&111101&111111&111112&\\
17& & &010001&010101&010102&010103&110103& &\\
18& & &001001&001011&001012&001013&011013& &\\
19& & &000101&100101&100102&100103&101103& &\\
20& & &000011&010011&010012&010013&010113& &\\
21& & &000002&100002&101002&101003&101013& &\\
22& & & &010002&010003&011003& & &\\
23& & & &001002&001003&001103& & &\\
24& & & &000102&000103&000113& & &\\
25& & & &000012&000013&100013& & &\\
26& & & &000003&100003&110003& & &\\
27& & & & &011101& & & &\\
28& & & & &001111& & & &\\
29& & & & &100111& & & &\\
30& & & & &110011& & & &\\
\end{tabular}
\caption{Symmetric chain decomposition of $P(5,4)$ with no taut chains}
\end{table}

\begin{table}
\centering
\begin{tabular}{c c c c c c c c c c c}
1& &000001&000002&000003&000004&000014&100014&110014&110114&\\
2& & &010001&010002&010003&010004&010014&010114& &\\
3& & &100001&100002&100003&100004&110004&111004& &\\
4& &010000&110000&110001&110002&110003&110103&110104&111104&\\
5& & &001001&001002&001003&001004&001104&101104& &\\
6& & & &011001&011002&011003&011004& & &\\
7& & & &101001&101002&101003&101004& & &\\
8& & &101000&111000&111001&111002&111003&111103& &\\
9& & &000101&000102&000103&000104&100104&100114& &\\
10& & & &010101&010102&010103&010104& & &\\
11& & &100100&100101&100102&100103&100113&101113& &\\
12& &000100&001100&001101&011101&011102&011103&011104&011114&\\
13& & & &101100&101101&101102&101103& & &\\
14& & & &110100&111100&111101&111102& & &\\
15& & &000011&000012&000013&001013&001014&101014& &\\
16& & & &100011&100012&100013&101013& & &\\
17& & & &110010&110011&110012&110013& & &\\
18&000000&001000&001010&001011&001012&001013&001113&001114&101114&111114\\
19& &000010&010010&010011&011011&011012&011013&011014&111014&\\
20& & &011000&011010&111010&111011&111012&111013& &\\
21& & & &000111&000112&000113&000114& & &\\
22& & &010100&010110&010111&010112&010113&011113& &\\
23& & &000110&100110&110110&110111&110112&110113& &\\
24& & & &011100&011110&011111&011112& & &\\
25& & & &001110&001111&101111&101112& & &\\
26& &100000&100010&101010&101110&111110&111111&111112&111113&\\
27& & & & &110101&110102& & & &\\
28& & & & &100111&100112& & & &\\
29& & & & &101011&101012& & & &\\
30& & & & &010012&010013& & & &\\
31& & & & &001102&001103& & & &
\end{tabular}
\caption{Symmetric chain decomposition of $P(5,5)$ with no taut chains}
\end{table}

From Tables 1,2,3, the proof of Theorem 3.1 is complete.

As Theorem~3.1 completely solves the question for $Q_k \times \textbf{n}$, one direction of further study would be to investigate other natural families of posets in a similar way using Theorems~3.2 and 3.3.

\end{document}